\documentclass[11pt]{article}
\usepackage{etex}
\usepackage{amssymb}
\usepackage[all,arc]{xy}
\usepackage{mathrsfs}
\usepackage{extarrows}
\usepackage{xcolor}
\usepackage{amsmath}
\usepackage{amsfonts,epsfig}
\usepackage{amssymb,bm}
\usepackage{graphicx}
\usepackage{subfig}
\usepackage{pgf,pgfarrows,pgfnodes,pgfautomata,pgfheaps}
\usepackage{appendix}

\usepackage{tabularx}  
\usepackage{booktabs}
\usepackage{multirow}
\usepackage{array,makecell}

\usepackage{algorithm}
\usepackage{algpseudocode}
\usepackage{pifont}

\usepackage{tikz}

\numberwithin{equation}{section}

\usepackage{amsmath,amsthm,amsfonts,amssymb,amscd,color}
\usepackage{hyperref}
\hypersetup{hypertex=true,
colorlinks=true,
linkcolor=blue,
anchorcolor=blue,
citecolor=blue}

\allowdisplaybreaks[4]
\newtheorem{theo}{Theorem}[section]
\newtheorem{lem}[theo]{Lemma}

\newtheorem{remark}{Remark}

\newtheorem{conj}[theo]{Conjecture}

\begin{document}

\newcommand{\E}{\mathbb{EX}}
\newcommand{\EX}{{\rm EX}}
\newcommand{\ex}{{\rm ex}}
\title{Generalized Tur\'an problem for a path and a clique}
\author{Xiaona Fang$^1$, Xiutao Zhu$^2$, Yaojun Chen$^{1,}$\thanks{Corresponding author: yaojunc@nju.edu.cn}\\
{\small $^1$Department of Mathematics, Nanjing University,}\\ {\small Nanjing, 210093, P.R. China}\\
{\small $^2$School of Mathematics, Nanjing University of Aeronautics and Astronautics,}\\ {\small Nanjing, 211106, P.R. China}
}
\date{}
\maketitle

\begin{abstract}
Let $\mathcal{H}$ be a family of graphs. The generalized Tur\'an number  $\ex(n, K_r, \mathcal{H})$ is the maximum number of copies of the clique $K_r$ in any $n$-vertex $\mathcal{H}$-free graph. In this paper, we determine the value of $\ex(n, K_r, \{P_k, K_m \} )$  for sufficiently large $n$ with an exceptional  case, and characterize all  corresponding extremal graphs, which generalizes and strengthens   the results of Katona and Xiao [EJC, 2024] on $\ex(n, K_2, \{P_k, K_m \} )$. For the exceptional case, we obtain a tight  upper bound for $\ex(n, K_r, \{P_k, K_m \} )$ that  confirms a conjecture on $\ex(n, K_2, \{P_k, K_m \} )$  posed by Katona and Xiao.


\end{abstract}

\section{Introduction}\label{sec1}
\hspace{1.5em}Let $\mathcal{H}$ be a family of graphs. A graph is called $\mathcal{H}$-free if it does not contain any member of $\mathcal{H}$ as a subgraph.
The Tur\'an number of $\mathcal{H}$, denoted by $\ex(n, \mathcal{H})$, is the maximum number of edges in any $n$-vertex $\mathcal{H}$-free graph. If we ask the maximum
under the condition that the graphs are connected, then denote it by
$\ex_{con}(n, \mathcal{H})$.
In 1941, Tur\'an \cite{turan1} proved that the unique extremal graph of $\ex(n, K_{k+1})$ is the balanced complete $k$-partite graph on $n$ vertices, denoted by $T(n,k)$. Since then, the Tur\'an problem has attracted a lot of attention and some related problems have been studied widely.

For convenience to state, we  define some graphs which will be mentioned many times. Let $K_n$, $P_n$, $I_n$ be a complete graph, a path and an empty graph on $n$ vertices, respectively. For two disjoint graphs $G$ and $H$, let $G \cup H$ denote the disjoint union of $G$ and $H$ and $kG$ denote the disjoint union of $k$ copies of $G$. Denote by $G\vee H$ the graph obtained from $G \cup H$ by adding edges between all vertices of $G$ and all vertices of $H$. Let $\delta_k=\lfloor\frac{k}{2}\rfloor-1$ and define
 $$	H_n(m,k)=  \begin{cases}
	T( \delta_k,m-2 ) \vee I_{ n- \delta_k} , & \!\! \text{ if } m\leq \delta_k +2 ,\\
		K_{\delta_k } \vee I_{n- \delta_k} ,	&\!\! \text{ if }  \delta_k +2 < m < k  \text{ and } k \text{ is even},  \\
        K_{\delta_k} \vee \left( I_{n-\delta_k-2 } \cup K_2 \right) ,	&\!\! \text{ if }  \delta_k+2< m < k \text{ and } k \text{ is odd}.
\end{cases}$$
When $m-2\le \delta_k\le 2m-5$ and $k$ is odd, let $H_n^-(m,k)$ be the graph obtained from $T( \delta_k,m-2 ) \vee \left( I_{n-\delta_k-2} \cup K_2 \right)$ by deleting an edge with one end vertex in $V(K_2)$ and another end vertex in a part of $T( \delta_k,m-2 )$ whose size is one.

The Tur\'an problem on a path can be tracked back to Erd\H{o}s-Gallai theorem.
\begin{theo}(Erd\H{o}s and Gallai \cite{Erdos1959})
For a path $P_k$, $\ex(n, P_k) \leq  \frac{n}{k-1} \binom{k-1}{2}$ and the equality holds if and only if $(k-1) | n$.
\end{theo}
Later, Faudree and Schelp \cite{Faudree} extended this result by determining the exact value of $\ex(n,P_k)$ for all $n$. In  a stronger case, Kopylov \cite{Kopylov}, and independently Balister et al. \cite{Balister}, studied the Tur\'an problem  of a path under the condition of connectivity.
\begin{theo}(Kopylov \cite{Kopylov}, Balister, Gy\H{o}ri, Lehel and Schelp \cite{Balister})\label{thl.5}
For $n>k\ge 4$,
\[\ex_{con}(n,P_k)=\max\left\{\binom{k-2}{2}+(n-k+2),\binom{\lceil k/2\rceil}{2}+\left\lfloor\frac{k-2}{2}\right\rfloor\left(n-\left\lceil\frac{k}{2}\right\rceil\right)\right\}.\]
Moreover, the extremal graph is either $K_1\vee(K_{k-3}\cup I_{n-k+2})$ or $H_n(k,k)$.
\end{theo}

Recently, Katona and Xiao \cite{katona2023} considered the Tur\'an problem when a path and a clique are forbidden at the same time, i.e., $\ex_{con}(n,\{P_k,K_m\})$ and $\ex(n,\{P_k,K_m\})$. They obtained the following two results, where $t(n,k)$ is the size of $T(n,k)$.
\begin{theo}(Katona and Xiao \cite{katona2023})\label{th1.3} For $k> m$ and sufficiently large $n$,
	$$\ex_{con}(n, \{P_k, K_m\} )=\delta_k  \cdot n+ t \left(\delta_k,m-2 \right)- \delta_k^2.$$
Moreover, $T( \delta_k,m-2 )\vee I_{n- \delta_k}$ is one of the extremal graphs.
\end{theo}

\begin{theo}(Katona and Xiao \cite{katona2023})\label{th1.4}
For $k >2m-1$ and sufficiently large $n$,
$$\ex (n, \{P_k, K_m\} )=\delta_k  \cdot n+ t \left( \delta_k,m-2 \right)- \delta_k  ^2 ,$$
	Moreover, $T( \delta_k,m-2 )\vee I_{n- \delta_k}$ is one of the extremal graphs.
\end{theo}

Note that in Theorem \ref{th1.4}, only the case $m< k\le 2m-1$ is left unsolved. So they proposed a conjecture as below.

\begin{conj}(Katona and Xiao \cite{katona2023})\label{conj}
 Let $m+1 \leq k \leq 2m-1$. If $k$ is odd and $(k-1)|n$, then $\frac{n}{k-1} \cdot T(k-1, m-1)$ gives the maximum while $T( \delta_k,m-2 ) \vee I_{ n- \delta_k}$ is the best for even $k$, for large $n$.
\end{conj}

Inspired by these results, Liu and Kang \cite{Kang} considered a more general problem. For any graph $H$, they determined the exact value of $\ex_{con}(n,\{P_k,H\})$ and the value of $\ex(n,\{P_k,H\})$ up to a constant term.

For a graph $T$ and a family $\mathcal{H}$, the generalized Tur\'an number, denoted by $\ex(n,T,\mathcal{H})$, is the maximum number of copies of $T$ in an $n$-vertex $\mathcal{H}$-free graph. Obviously, when $T=K_2$, $\ex(n,T,\mathcal{H})=\ex(n,\mathcal{H})$. And $\ex_{con}(n,T,\mathcal{H})$ is defined similar to $\ex_{con}(n,\mathcal{H})$.   This concept was first defined by Alon and Shikhelman \cite{alon2016} and has been a very active problem in the last decade. For the path $P_k$, its generalized Tur\'an number is also studied widely. Let $N_r(G)$ denote the number of copies of $K_r$ in $G$. In 2017, Luo \cite{luo2017} obtained the generalized version of Theorem \ref{thl.5}.
\begin{theo}(Luo \cite{luo2017})\label{luo}
For $n\ge k\ge 4$,
\[\ex_{con}(n,K_r,P_k)=\max\{N_r(K_1\vee(K_{k-3}\cup I_{n-k+2})), N_r(H_n(k,k))\}.\]
Moreover, the extremal graph is either $K_1\vee(K_{k-3}\cup I_{n-k+2})$ or $H_n(k,k)$.
\end{theo}
As a corollary of Theorem \ref{luo}, one can easily obtained an Erd\H{o}s-Gallai type upper bound, i.e., $\ex(n,K_r,P_k)\le \frac{n}{k-1}\binom{k-1}{r}$. Recently, Chakraborti and Chen \cite{DQ} obtained the exact value of $\ex(n,K_r,P_k)$.
\vskip 2mm
 In this paper, we focus on the generalized Tur\'an numbers $\ex_{con} (n, K_r, \{P_k, K_m\})$ and $\ex (n, K_r, \{P_k, K_m\})$,
and try to extend Katona and Xiao's  results on classical Tur\'an number and solve their problem. Let $S_{a,n-a}$ be the double star obtained from stars $S_{a}$ and $S_{n-a}$ by joining their centers. The main results of this paper are as follows.
\begin{theo}\label{gener}
  	For $k>m$, $r\leq \min\{m-1, \delta_k+1 \}$ and sufficiently large $n$,
   $$\ex_{con} (n, K_r, \{P_k, K_m\}) = N_r(H_n(m,k)).$$
The extremal graphs are $H_n(m,k)$, $H_n^-(m,k)$ {\rm (}$m-2\le \delta_k\le 2m-5$ and $k$ is odd{\rm )} and $S_{a,n-a}$ {\rm (}$k=5, m=3$ and $r=2${\rm )}.
\end{theo}

\begin{theo}\label{nonconnect}
Let $k>m$ and $r\leq \min\{m-1, \delta_k+1 \}$.  Then
\vskip 2mm	
{\rm (1)} $\ex (n, K_r, \{P_k, K_m\}) = N_r(H_n(m,k))$  if $N_{r-1} ( T(\delta_k, m-2)) >\frac{ N_r(T(k-1,m-1))}{k-1}$ and $n$ is large,  and the extremal graphs are $H_n(m,k)$, $H_n^-(m,k)$ {\rm (}$m-2\le \delta_k\le 2m-5$ and $k$ is odd{\rm )} and $S_{a,n-a}$ {\rm (}$k=5, m=3$ and $r=2${\rm )}; 
\vskip 2mm	
{\rm (2)} $\ex (n, K_r, \{P_k, K_m\}) \leq \frac{N_r(T(k-1, m-1))}{k-1} n$ if $N_{r-1} ( T(\delta_k, m-2)) \leq \frac{ N_r(T(k-1,m-1))}{k-1}$.  When $(k-1) | n$, the equality holds and $\frac{n}{k-1} \cdot T(k-1, m-1)$ is an extremal graph.
\end{theo}

\begin{remark}
	When $r=2$, Theorem \ref{gener} implies Theorem \ref{th1.3} (Note that Theorem \ref{th1.3} is not true when  $k$ is odd and $\delta_k<m-2$, because  the graph  $K_{\delta_k}\vee (I_{n- \delta_k-2}\cup K_2)$ has more edges than $T( \delta_k,m-2 )\vee I_{n- \delta_k}$, regard $T(\delta_k,m-2)$  as $K_{\delta_k}$). 
When $2m-1 < k$, we have $\frac{t(k-1,m-1)}{k-1} < \frac{1}{k-1} \binom{m-1}{2} (  \frac{k-1}{m-1} ) ^2 < \frac{k-3}{2}$. Thus $N_{1} ( T(\delta_k, m-2)) >\frac{ N_2(T(k-1,m-1))}{k-1}$, and then Theorem \ref{nonconnect}(1) implies Theorem \ref{th1.4}. Moreover,  all extremal graphs of Theorems \ref{th1.3} and \ref{th1.4} are determined in Theorems \ref{gener} and \ref{nonconnect},  which strengthens the two results.
\end{remark}

\begin{remark}
	When $k$ is odd and $m+1 \leq k \leq 2m-1$, we have $1<\frac{k-1}{m-1} \leq 2$, and then $t(k-1, m-1) \geq \binom{k-1}{2} - \frac{k-1}{2} =\delta_k \cdot (k-1) $, which implies $N_{1} ( T(\delta_k, m-2))  \leq \frac{ N_2(T(k-1,m-1))}{k-1}$. When $k$ is even and $k>m$, we have $\delta_k \cdot (k-1) = \frac{(k-2) (k-1)}{2} =e(K_{k-1}) >t(k-1,m-1) $, which implies $N_{1} ( T(\delta_k, m-2)) >\frac{ N_2(T(k-1,m-1))}{k-1}$. Thus Theorem \ref{nonconnect} implies that Conjecture \ref{conj} is true by taking $r=2$.
\end{remark}

 The remainder of this paper is organized as follows. In Section \ref{sec2}, we give some preliminaries. To prove Theorems \ref{gener} and  \ref{nonconnect}, we  obtain a version of stability result for $\{P_k, K_m\}$, that is, Theorem \ref{max}, which characterizes the structure of a connected $\{P_k, K_m\}$-free graph with large minimum degree in Section \ref{sec3}.  In Section \ref{sec4}, we give the proofs of Theorems \ref{gener} and  \ref{nonconnect} using Theorem \ref{max}.

\section{Preliminaries}\label{sec2}

\hspace{1.5em}A family of $k$ internally disjoint $(v,Y)$-paths whose terminal vertices are distinct in $Y$ is referred to as a $k$-fan from $v$ to $Y$.

\begin{lem}(Bondy and Murty \cite{bondy})\label{lem}
Let $G$ be a $k$-connected graph, let $v$ be a vertex of $G$ and $Y \subseteq V(G) \setminus \{v\}$ be a set of at least $k$ vertices of $G$. Then there exists a $k$-fan in $G$ from $v$ to $Y$.
\end{lem}

\begin{lem}(Erd\H os and Gallai \cite{Erdos1959})\label{lem2}
	Let $G$ be a $2$-connected graph and $x, y$ be two given vertices. If every vertex other than $x, y$ has degree at least $k$ in $G$, then there is an $(x, y)$-path of length at least $k$.
\end{lem}

For a vertex $v \in V(G)$ and a subgraph $H \subseteq G$, let $G - H$ denote the subgraph induced by $V (G)\setminus V (H)$, and $N_H(v)$ denote the set of vertices in $H$ which are adjacent to $v$ and  $d_H(v) = |N_H(v)|$. We simply write $N(v)=N_G(v)$ and use $\delta(G)$ to denote the minimum degree of $G$.
Let $p(G)$ be the order of a longest path in $G$ and $P$ is hamiltonian if $|P|=|G|$. We call $P$ a strong dominating path if $N(v)\subseteq V(P)$ for any $v\in V(G-P)$ and a strong dominating cycle is defined similarly. 

Let $\sigma_3(G) = \max \{d(u) + d(v) + d(w) \mid  \{u, v, w\} \text{ is an independent set in } G\}.$

\begin{lem}(Saito \cite{saito1999}) \label{lem5}
	Suppose $G$ is a $2$-connected graph of order $n$. Then either $G$ contains a strong dominating cycle or $p(G) \geq
	\min \{n,\sigma_3(G) - 1\}$.
\end{lem}

\begin{lem}(Katona and Xiao \cite{katona2023}) \label{lem3}
	Let $G$ be a connected graph on $k$ vertices with no Hamiltonian path, but with a path
	$P =v_1 v_2\cdots v_{k-1}$ on $k-1$ vertices. Suppose the vertex $u \in  V(G-P)$ has degree $s$, that is
	$N_P (u) =\{v_{i_1}, v_{i_2},\dots, v_{i_s} \}$, $i_1 < i_2 < \cdots < i_s$. Then
	
	{\rm (i)} $s \leq \delta_k$;
	
	{\rm (ii)} there are no edges of the form $v_{i_j+1} v_{i_r+1}$, $v_{i_j-1} v_{i_r-1}$, $v_1 v_{i_r+1}$ or $v_{i_r-1} v_{k-1}$ $(1 \leq  j < r \leq s)$.
\end{lem}

\begin{lem}(Zykov \cite{zykov1949}) \label{ThKr}
	For $2\leq r<m \leq n$, $\ex(n,K_r,K_m)=N_r(T(n,m-1))$, and  $T(n,m-1)$ is the unique extremal graph.
\end{lem}


\section{$\{P_k,K_m\}$-free graphs with large minimum degree}\label{sec3}

\hspace{1.5em}In this section, we characterize the structure of $\{P_k,K_m\}$-free graph $G$ with $\delta(G)\geq \delta_k$. An end block of $G$ is one having exactly one cut vertex. To state our results, we firstly define several graphs.  For $\delta_{k} \leq m-2$, let $G_1(n,k)=K_1\vee tK_{\delta_k}$ where $n=1+t\delta_{k}$ and $t\geq 1$,
$G_2(n,k)$ ($k$ is odd) be a graph on $n$ vertices obtained by joining the centres of $G_1(n_1,k)$ and $G_1(n_2,k)$, where $n=n_1+n_2$, and $G_3(n,k)$ ($k$ is odd) be a graph on $n$ vertices obtained from $G_1(n-1,k)$ by replacing a block $K_{\delta_k+1 }$ with a  $K_m$-free block $B$ of order $\delta_k+2$ such that the resulting graph has minimum degree at least $\delta_k$. 
Let $G_4(n)$ ($G_5(n)$) be an $n$-vertex graph obtained by joining the centre of $G_1(n_1,7)$ to all leaves of $K_{1,n_2-2}$ (all vetices of $K_{1,n_2-2}$), where $n_2\geq 4$, $n=n_1+n_2-1$ (See Figure \ref{graph}).

\begin{figure}[h]
	\centering
	\begin{tikzpicture}
		\filldraw[fill=black!8] plot[smooth, tension = 0.9]
		coordinates{(0,0) (60:1) (90:1.6) (120:1) (0,0)};
		\filldraw[fill=black!8] plot[smooth, tension = 0.9]
		coordinates{(0,0) (180:1) (210:1.6) (240:1) (0,0)};
		\filldraw[fill=black!8] plot[smooth, tension = 0.9]
		coordinates{(0,0) (-60:1) (-30:1.6) (0:1) (0,0)};
		
		\filldraw
		(0,0)
		++(-90:0.8) circle (0.02) ++(0:0.1) circle (0.02) ++(180:0.2) circle (0.02);
		
		\filldraw
		(0,0) circle (0.03);
		
		\node at (0,0.8) {\footnotesize  $K_{\delta_k+1}$};
		
		\node at (0,-1.7) {$G_1(n,k)$};
	\end{tikzpicture}\qquad
	\begin{tikzpicture}
		\filldraw[fill=black!8] plot[smooth, tension = 0.9]
		coordinates{(0,0) (100:1) (130:1.6) (160:1) (0,0)};
		\filldraw[fill=black!8] plot[smooth, tension = 0.9]
		coordinates{(0,0) (200:1) (230:1.6) (260:1) (0,0)};
		
		\draw[solid] (0,0) -- (1,0);
		
		\filldraw[fill=black!8] plot[smooth, tension = 0.9]
		coordinates{(1,0)  ({1+cos(20)},{sin(20)})  ({1+1.6*cos(50)},{1.6*sin(50)})  ({1+cos(80)},{sin(80)}) (1,0) };
		\filldraw[fill=black!8] plot[smooth, tension = 0.9]
		coordinates{(1,0)  ({1+cos(-20)},{sin(-20)})  ({1+1.6*cos(-50)},{1.6*sin(-50)})  ({1+cos(-80)},{sin(-80)}) (1,0) };

		\filldraw
		(0,0) ++(180:0.9) circle (0.02) ++(90:0.1) circle (0.02) ++(-90:0.2) circle (0.02);
		\filldraw
		(0,0) ++(0:1.9) circle (0.02) ++(90:0.1) circle (0.02) ++(-90:0.2) circle (0.02);
		
		\filldraw  (0,0) circle (0.05);
		\filldraw  (1,0) circle (0.05);
		
		\node at (0.5,-1.95) {$G_2(n,k)$};
	\end{tikzpicture}\qquad
	\begin{tikzpicture}
		\filldraw[fill=black!8] plot[smooth, tension = 0.9]
		coordinates{(0,0) (60:1) (90:1.6) (120:1) (0,0)};
		\filldraw[fill=black!8] plot[smooth, tension = 0.9]
		coordinates{(0,0) (180:1) (210:1.6) (240:1) (0,0)};
		\filldraw[fill=red!8] plot[smooth, tension = 0.9]
		coordinates{(0,0) (-60:1) (-30:1.6) (0:1) (0,0)};
		
		\filldraw
		(0,0)
		++(150:0.8) circle (0.02) ++(60:0.1) circle (0.02) ++(-120:0.2) circle (0.02);
		
		\filldraw
		(0,0) circle (0.05);
		
		\node at (0.8,-0.5) {\footnotesize $B$};
		
		\node at (0,-1.7) {$G_3(n,k)$};
	\end{tikzpicture}\\[15pt]
\begin{tikzpicture}	
	\filldraw
	(-0.5,0) circle (0.03)
	--(0,-1.4) circle (0.03)
	--(-0.3,-1.4) circle (0.03)
	--(-0.5,0) circle (0.03)
	--(-1,-1.4) circle (0.03)
	--(-1.3,-1.4) circle (0.03)
	--(-0.5,0) circle (0.03)
	--(0.3,-1.4) circle (0.03)
	--(0.5,0) circle (0.03)
	--(0.6,-1.4) circle (0.03)
	--(-0.5,0) circle (0.03)
	--(1.3,-1.4) circle (0.03)
	--(0.5,0) circle (0.03);
	
	\filldraw
	(-0.65,-1.4)
	circle (0.015) ++(0:0.1) circle (0.015) ++(180:0.2) circle (0.015);
	
	\filldraw
	(0.93,-1.4)
	circle (0.015) ++(0:0.1) circle (0.015) ++(180:0.2) circle (0.015);
	
	\node at (0,-2) {$G_4(n)$};
\end{tikzpicture}\qquad\qquad
\begin{tikzpicture}
	\filldraw
	(-0.5,0) circle (0.03)
	--(0,-1.4) circle (0.03)
	--(-0.3,-1.4) circle (0.03)
	--(-0.5,0) circle (0.03)
	--(-1,-1.4) circle (0.03)
	--(-1.3,-1.4) circle (0.03)
	--(-0.5,0) circle (0.03)
	--(0.3,-1.4) circle (0.03)
	--(0.5,0) circle (0.03)
	--(0.6,-1.4) circle (0.03)
	--(-0.5,0) circle (0.03)
	--(1.3,-1.4) circle (0.03)
	--(0.5,0) circle (0.03);
	
	\filldraw
	(-0.5,0) --(0.5,0) ;

	\filldraw
	(-0.65,-1.4)
	circle (0.015) ++(0:0.1) circle (0.015) ++(180:0.2) circle (0.015);
	
	\filldraw
	(0.93,-1.4)
	circle (0.015) ++(0:0.1) circle (0.015) ++(180:0.2) circle (0.015);
	
	\node at (0,-2) {$G_5(n)$};  	
\end{tikzpicture}

	\caption{The graphs $G_i(n,k)$ ($i=1,2,3$) and $G_i(n)$ ($i=4,5$)}
	\label{graph}
\end{figure}
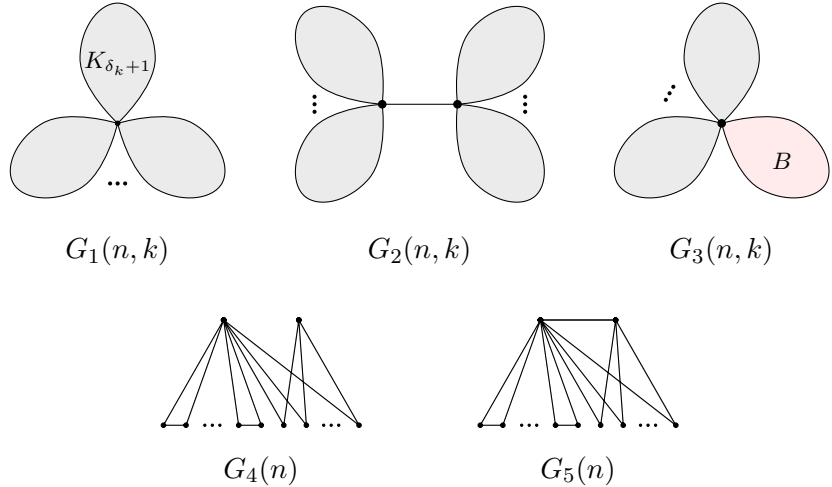
\begin{theo}\label{max}
	Let $G$ be a $\{P_k, K_m\}$-free connected graph with $\delta(G)\geq \delta_k $ and $|G|=n\geq k$. Then one of the following holds:
	
	{\rm (1)} $G\subseteq \begin{cases}
		H \vee I_{n-\delta_k},	& \text{ if } k \text{ is even},  \\
		H \vee \left( I_{n-\delta_k-2} \cup K_2 \right),	& \text{ if } k \text{ is odd},
	\end{cases}$ \\ where $H$ is a $K_{m-1}$-free graph on $\delta_k$ vertices;
	
	{\rm (2)} $m\ge \delta_k+2$ and $G\cong G_1(n,k)$, $G_2(n,k)$ or $G_3(n,k)$;
	
	{\rm (3)} $k=7$, $m\geq 4$ and $G\cong G_4(n)$ or $G_5(n)$;
	
	{\rm (4)} $k=9$, $m\geq 4$ and $G=I_2 \vee \frac{n-2}{2} K_2$ or $m\geq 5$ and $G=K_2 \vee \frac{n-2}{2} K_2$.
\end{theo}

In order to prove Theorem \ref{max}, we need some additional notations.   
If $u$ and $v$ are two vertices of a path $P$, we use $uPv$ to denote the segment of $P$ from $u$ to $v$. Denote by $\overrightarrow{P}$ the path $P$ with a given orientation and $\overleftarrow{P}$ the path with the inverse orientation. We use $u^{+i}$ and $u^{-i}$ to denote the $i$th successor and predecessor of $u$ on $P$ along a given orientation, respectively. In particular, $u^+=u^{+1}$ and $u^-=u^{-1}$.
\vskip 2mm
Before  starting to prove Theorem \ref{max}, we need  the following two lemmas.  

\begin{lem}\label{lem_block}
	Let $G$ be an $n$-vertex $\{P_k, K_m\}$-free connected graph but not $2$-connected. If any vertex except the cut vertex in the end blocks  has degree at least $\delta_k$, then $G\cong G_1(n,k)$, $G_2(n,k)$, $G_3(n,k)$ with $m\ge \delta_k+2$, or $G \cong G_4(n)$, $G_5(n)$ with  $k=7$, $m\geq 4$.
\end{lem}
\begin{proof}
 By Lemma \ref{lem2}, each end block contains a path of length at least $\delta_k$ starting from the cut vertex. Furthermore, either all end blocks share one cut vertex or  there are two end blocks which have no common vertices. For the second case, suppose $B_1$ and $B_2$ are two end blocks that thay have no common vertices. Let $v_i$ be the cut vertex in $B_i$ and $P_i$ be the longest path starting from $v_i$, $i=1,2$, and $Q$ a longest path  connecting $v_1$ and $v_2$. Then $P_1v_1Qv_2P_2$ is a  path of length at least $2\delta_k+|Q|-1$. Since $G$ is $P_k$-free, we can deduce $|Q|= 2$. Thus, the structure of $G$ can only be as follows. (see Figure \ref{fig1})
\vskip 1mm
\noindent	 \textbf{Structure 1}: all end blocks share one cut vertex, and the longest path starting from the cut vertex in each end block is $P_{\delta_k+1 }$ or $P_{\delta_k+2}$. Furthermore, at most one end block has a path $P_{\delta_k+2}$ starting from the cut vertex, and if this happens, then $k$ is odd; 
\vskip 1mm	
\noindent \textbf{Structure 2}:  $G$ has only one non-end block and the block is a $K_2$,  and the longest path starting from the cut vertex in each end block is $P_{\delta_k+1 }$. Moreover,  $k$ is odd.	
	\vskip 2mm
	\begin{figure}[h]
	\centering
	\begin{tikzpicture}
		\draw[solid] plot[smooth, tension = 0.9]
		coordinates{(0,0) (60:1) (90:1.6) (120:1) (0,0)};
		\draw[solid] plot[smooth, tension = 0.9]
		coordinates{(0,0) (180:1) (210:1.6) (240:1) (0,0)};
		\draw[solid] plot[smooth, tension = 0.9]
		coordinates{(0,0) (-60:1) (-30:1.6) (0:1) (0,0)};
		
		\draw[blue] (0,0) -- (90:1.3);
		\draw[blue] (0,0) -- (210:1.3);
		\draw[red] (0,0) -- (-30:1.3);
		
		\node[blue] at (-0.94,1.2) {\footnotesize  $P_{\delta_k+1}$};
		\node[blue] at (-1.8,-1) {\footnotesize  $P_{\delta_k+1}$};
		\node[red] at (2.5,-1) {\footnotesize  $P_{\delta_k+1}$ or $P_{\delta_k+2}$};
		
		\filldraw
		(0,0)
		++(150:0.8) circle (0.02) ++(60:0.1) circle (0.02) ++(-120:0.2) circle (0.02);
		
		\filldraw  (0,0) circle (0.05);
		
		\node at (0,-1.85) {(1)};
	\end{tikzpicture}\qquad\quad
\begin{tikzpicture}
	\draw[solid] plot[smooth, tension = 0.9]
	coordinates{(0,0) (100:1) (130:1.6) (160:1) (0,0)};
	\draw[solid] plot[smooth, tension = 0.9]
	coordinates{(0,0) (200:1) (230:1.6) (260:1) (0,0)};
	
	\draw[solid] (0,0) -- (1,0);
	
	\draw[solid] plot[smooth, tension = 1]
	coordinates{(1,0)  ({1+cos(20)},{sin(20)})  ({1+1.6*cos(50)},{1.6*sin(50)})  ({1+cos(80)},{sin(80)}) (1,0) };
	\draw[solid] plot[smooth, tension = 1]
	coordinates{(1,0)  ({1+cos(-20)},{sin(-20)})  ({1+1.6*cos(-50)},{1.6*sin(-50)})  ({1+cos(-80)},{sin(-80)}) (1,0) };
	
	\draw[blue] (0,0) -- (130:1.3);
	\draw[blue] (0,0) -- (230:1.3);
	\draw[blue] (1,0) -- ++(50:1.3);
	\draw[blue] (1,0) -- ++(-50:1.3);
	
	\node[blue] at (-1.68,1) {\footnotesize  $P_{\delta_k+1}$};
\node[blue] at (2.65,1) {\footnotesize  $P_{\delta_k+1}$};
\node[blue] at (-1.68,-1) {\footnotesize  $P_{\delta_k+1}$};
\node[blue] at (2.65,-1) {\footnotesize  $P_{\delta_k+1}$};
	
	\filldraw
	(0,0) ++(180:0.9) circle (0.02) ++(90:0.1) circle (0.02) ++(-90:0.2) circle (0.02);
	\filldraw
	(0,0) ++(0:1.9) circle (0.02) ++(90:0.1) circle (0.02) ++(-90:0.2) circle (0.02);
	
	\filldraw  (0,0) circle (0.05);
	\filldraw  (1,0) circle (0.05);
	
	\node at (0.5,-1.95) {(2)};
	\end{tikzpicture}
	\caption{The structure of $G$}
	\label{fig1}
\end{figure}
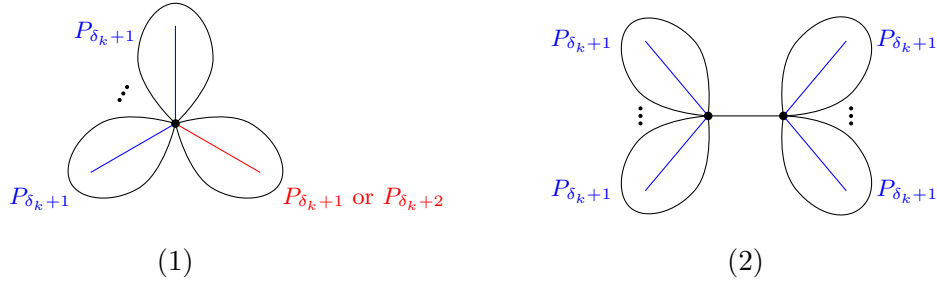	

By this two structures, we also have the following claims.

\vskip 1mm
\noindent\textbf{Claim 1.} If the longest path starting from the cut vertex in the end block $B$ is $P_{\delta_k+1 }$, then $B= K_{\delta_k+1}$.
\begin{proof}
Let $P$ be the longest path starting from the cut vertex $v$ in $B$ and the other end is $u$. Since $d(u)\geq \delta_k$, $u$ is adjacent to all other vertices of $P$. If there is a vertex $w\notin V(P)$, then since $B$ is $2$-connected, there is a $2$-fan from $w$ to $V(P)$ by Lemma \ref{lem}. Let $z(\neq v)$ be the vertex on $P$ connected to $w$ by a path $Q \subseteq G - P$. Then $v \overrightarrow{P} z^- u \overleftarrow{P}z Q w$ will be a longer path starting from $v$,  a contradiction. Thus $V(B)=V(P)$ and Claim 1 follows by the condition of degree.
\end{proof}
	
	\noindent\textbf{Claim 2.} If some end block $B$ has a path $P_{\delta_k+2}$ starting from the cut vertex $v$ in $B$, then $G\cong G_3(n,k)$ with $m\ge \delta_k+2$, or $G \cong G_4(n)$, $G_5(n)$ with $k=7$, $m\geq 4$.	
	
 \begin{proof}
By the assumption,  the structure of $G$ is as shown in Figure 2(1). If $|B|=\delta_k+2$, then $B$ is a $K_m$-free block with $d(u)\geq \delta_k$ for any $u \in B-v$. The other blocks are all $K_{\delta_k+1}$ by Claim 1, and hence $G \cong G_3(n,k)$ and $m\ge \delta_k+2$.

Next we may assume $|B|> \delta_k+2$ and $P$ is a longest path of $B$ starting from the cut vertex $v$ in $B$ and ending at $u$.  By the maximality of $P$,  $N(u)\subseteq V(P)$.  Let $w$ be any vertex not in $P$. Since $B$ is $2$-connected, by Lemma \ref{lem}, there is a $2$-fan from $w$ to $V(P)$. Assume that $x\not=v$ in $P$ and $Q$ is a path connecting $w$ and $x$ in a 2-fan from $w$ to $V(P)$.  Then $x^-\notin N(u)$, for otherwise $v \overrightarrow{P} x^- u \overleftarrow{P}xQ w$
is a  path  longer than $P$, a contradiction. Therefore, $N(u)=V(P)-\{u,x^-\}$ and the terminal vertices of the 2-fan from $w$ to $V(P)$  must be $\{v, x\}$.

Now let $wQ_1v$, $wQ_2x$ be the $2$-fan from $w$ to $V(P)$. If $|vPx|\geq 4$, then we can see $vQ_1 w Q_2 x \overrightarrow{P} u v^+ \overrightarrow{P} x^-$ is a path longer than $P$, a contradiction. By the maximality of $P$, $x\not=v^+$ and hence $x=v^{+2}$.
 Moreover, $x^-$ has no neighbor in $V(B-P)$, otherwise $vu\overleftarrow{P}x^-x'$ is a path longer than $P$ for some $x'\in V(B-P)$.  Thus, note that $d(x^-) \geq \delta_k$, we have $N(x^-)=V(P)-\{x^-,u\}$.  If $|x^-Pu|\geq 4$, then $vQ_1 w Q_2 xx^-x^+ \overrightarrow{P} u$ is a path longer than $P$, again a contradiction. 
 Hence, $u=x^+$, that is $|P|=\delta_k+2=4$.
 Since there is at least a block $K_{\delta_k+1}=K_3$ in $G$ by Claim 1 (see Figure \ref{fig1}), we have $m\geq 4$ and $k=7$. By the maximality of $P=P_4$,  we have $N(w)=\{v,x\}$ for any $w\in V(G-P)$. If $v$ is not adjacent to $x$, then $G\cong G_4(n)$ and $G\cong G_5(n)$ otherwise. 

	Hence, Claim 2 holds.
\end{proof}	
	Now, if the structure of $G$ is as shown in Figure 2(1), then by Claims 1 and  2, $G\cong G_1(n,k)$, $G_3(n,k)$ and $m\ge \delta_k+2$, or $G \cong G_4(n)$, $G_5(n)$ with $k=7$, $m\geq 4$. If the structure of $G$ is as shown in Figure 2(2), then $G\cong G_2(n,k)$ and $k$ is odd. We complete the proof of this lemma.
\end{proof}

\begin{lem}\label{lem_dominate}
	Let $G$ be a 2-connected  $\{P_k, K_m\}$-free graph of order $n\geq k$. If $G$ contains a strong dominating path $P$ such that $d(u)\geq \delta_k$ for any $u\in V(G-P)$, then
 $$G\subseteq \begin{cases}
		H \vee I_{n-\delta_k},	& \text{ if } k \text{ is even},  \\
		H \vee \left( I_{n-\delta_k-2} \cup K_2 \right),	& \text{ if } k \text{ is odd},
	\end{cases}$$  where $H$ is a $K_{m-1}$-free graph on $\delta_k$ vertices.
\end{lem}

\begin{proof}
	Let $P=v_1v_2\cdots v_{p}$ be a longest strong dominating path satisfying the condition. By the maximality of $P$, $N_{G-P} (v_1) = N_{G-P} (v_p) = \varnothing $ and any vertex of $G-P$ has no consecutive neighbors on $P$. Thus we have
	\begin{equation}\label{eq1}
	k-1\geq	|P| \geq  2 \cdot \delta_k  + 1,
	\end{equation}
which implies $|P|=k-2$ or $k-1$.

For the case $|P|=k-2$, it is easy to get that $k$ is odd and $N(u)=\{v_2,v_4,\ldots,v_{k-3}\}$ for every $u\in V(G-P)$.  Since $G$ is $K_m$-free, then $G[N(u)]$ is $K_{m-1}$-free. Moreover, $V(P)-N(u)$ is also independent by Lemma \ref{lem3}. Thus, $G\subseteq H \vee I_{n-\delta_k}$.

Now we assume that $|P|=k-1$. If $k$ is even, then $N_P(u)=\{v_2,v_4,\dots, v_{k-2}\}$. Since $G$ is $K_m$-free, $G[N(u)]$ is a $K_{m-1}$-free graph on $\delta_k$ vertices. By Lemma \ref{lem3}, $\{v_1,v_3,  \dots, v_{k-1}\}$ is an independent set. Thus $G \subseteq H \vee I_{n-\delta_k}$. If $k$ is odd, then either
	
	(i) $N_P(u) =\{ v_2, v_4, \dots, v_{k-3}\}$ (or its symmetric version $\{ v_3, v_5, \dots, v_{k-2}\}$), or
	
	(ii) $N_P(u) =\{ v_2, \dots, v_{2\ell}, v_{2\ell+3},  \dots, v_{k-2}\}$ ($1\leq \ell \leq \frac{k-5}{2}$).\\
If $k=5$, then $|P|=4$ and so $d(u)=1$,  which contradicts that $G$ is not 2-connected. Thus we have $k\geq 7$. Suppose that  $N_P(u_1)\not=N_P(u_2)$ for some $u_1 , u_2 \in V(G-P)$.  If  $u_1, u_2$ are both of type (i), then by symmetry we may assume $N(u_1)=\{ v_2, v_4, \dots, v_{k-3}\}$ and hence $N(u_2)= \{ v_3, v_5, \dots, v_{k-2}\} $,  and  $v_1 v_2 u_1 v_{k-3} v_{k-2} u_2 v_3 \overrightarrow{P} v_{k-4}$ is a  $P_k$ in $G$ (see Figure \ref{fig3}(1)), a contradiction. If $u_1$ is of type (i) and $u_2$ is of type (ii), then assume $N(u_1)=\{ v_2, v_4, \dots, v_{k-3}\}$ and $v_3 \overrightarrow{P} v_{k-3}u_1 v_2 u_2  v_{k-2} v_{k-1}$ is a  $P_k$ in $G$ (see Figure \ref{fig3}(2)), a contradiction.
If $u_1, u_2$ are both of type (ii), we may assume that $N(u_1)=\{ v_2, \dots, v_{2\ell_1}, v_{2\ell_1+3},  \dots, v_{k-2}\}$ and $N(u_2)= \{ v_2, \dots, v_{2\ell_2}, v_{2\ell_2+3},  \dots, v_{k-2}\} $ with $\ell_1 <\ell_2$, then $G$ has  a  $P_k=v_3 \overrightarrow{P} v_{2\ell_1+2} u_2 v_2 u_1  v_{2\ell_1+3} \overrightarrow{P} v_{k-1}$ (see Figure \ref{fig3}(3)), again a contradiction. Thus we have $N_P(u_1)=N_P(u_2)$ for any $u_1 , u_2 \in V(G-P)$. By Lemma \ref{lem3}, there is no other edge in $G-N(u)$ except for one in path $P$. Since $G$ is $K_m$-free, $G[N_P(u)]$ is $K_{m-1}$-free. Thus $G \subseteq H \vee \left( I_{n-\delta_k-2} \cup K_2 \right)$. 
	
	Therefore, Lemma \ref{lem_dominate} holds.
\end{proof}	
	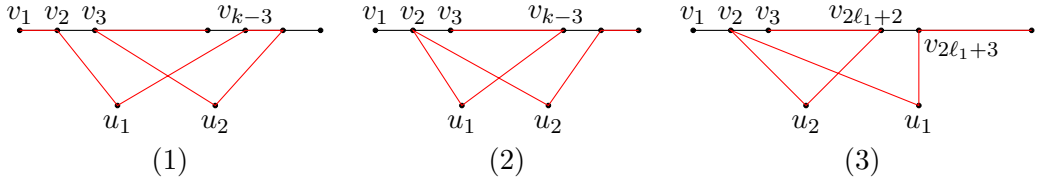
\begin{figure}[h]
		\centering
		\begin{tikzpicture}
			
			\filldraw
			(0,0) circle (0.03)
			--++(0.5,0) circle (0.03)
			--++(0.5,0) circle (0.03)
			--++(1.5,0) circle (0.03)
			--++(0.5,0) circle (0.03)
			--++(0.5,0) circle (0.03)
			--++(0.5,0) circle (0.03)
			(1.3,-1) circle (0.03)
			(2.6,-1) circle (0.03) ;
			
			\draw[red] (0,0) --(0.5,0)--(1.3,-1) --(3,0)--(3.5,0)--(2.6,-1) --(1,0)--(2.5,0);
			
			\node at (0,0.2) {$v_1$};
			\node at (0.5,0.2) {$v_2$};
			\node at (1,0.2) {$v_3$};
			\node at (3,0.2) {$v_{k-3}$};
			
			\node at (1.3,-1.25) {$u_1$};
			\node at (2.6,-1.25) {$u_2$};
			
			\node at (2,-1.75) {(1)};
		\end{tikzpicture}\quad
		\begin{tikzpicture}
			\filldraw
			(0,0) circle (0.03)
			--(0.5,0) circle (0.03)
			--++(0.5,0) circle (0.03)
			--++(1.5,0) circle (0.03)
			--++(0.5,0) circle (0.03)
			--(3.5,0) circle (0.03)
			(1.15,-1) circle (0.03)
			(2.3,-1) circle (0.03);
			
			\draw[red] (1,0)--(2.5,0)--( 1.15,-1)--(0.5,0)--(2.3,-1)--(3,0)--(3.5,0);
			
			\node at (0,0.2) {$v_1$};
			\node at (0.5,0.2) {$v_2$};
			\node at (1,0.2) {$v_3$};
			\node at (2.4,0.2) {$v_{k-3}$};
			
			\node at (1.15,-1.25) {$u_1$};
			\node at (2.3,-1.25) {$u_2$};
			
			\node at (1.75,-1.75) {(2)};  	
		\end{tikzpicture}\quad
		\begin{tikzpicture}
			\filldraw
			(0,0) circle (0.03)
			--(0.5,0) circle (0.03)
			--++(0.5,0) circle (0.03)
			--++(1.5,0) circle (0.03)
			--++(0.5,0) circle (0.03)
			--(4.5,0) circle (0.03)
			(1.5,-1) circle (0.03)
			(3,-1) circle (0.03);
			
			\draw[red] (1,0)--(2.5,0) --(1.5,-1)--(0.5,0)--(3,-1)--(3,0)--(4.5,0);	
			
			\node at (0,0.2) {$v_1$};
			\node at (0.5,0.2) {$v_2$};
			\node at (1,0.2) {$v_3$};
			\node at (2.3,0.2) {$v_{2\ell_1+2}$};
			\node at (3.57,-0.25) {$v_{2\ell_1+3}$};
			
			\node at (3,-1.25) {$u_1$};
			\node at (1.5,-1.25) {$u_2$};
			
			\node at (2.25,-1.75) {(3)};  	
		\end{tikzpicture}
		\caption{Illustration of a path $P_{k}$}
		\label{fig3}
	\end{figure}
	
	\vskip 0.2em

\noindent\textbf{Proof of Theorem \ref{max}}.
If $G$ is not $2$-connected, then by Lemma \ref{lem_block}, (2) and (3) hold. If $G$ is $2$-connected and has a strong dominating path, then by Lemma \ref{lem_dominate}, (1) holds. Therefore, we are left to consider the case when $G$ is 2-connected but has no strong dominating path. 

It is clear that $G$ has no strong dominating path implies that $G$ has no dominating cycle. By Lemma \ref{lem5}, $p(G) \geq
\min \{n,\sigma_3(G) - 1\}$. Since $G$ is $P_k$-free, we have
$$ k > \sigma_3 (G) - 1 \geq 3 \delta_k -1,$$
which implies $k \in \{ 3,4,5,6,7,9\}$. Let $P=v_1v_2 \cdots v_p$ be a longest path in $G$. Since $G$ has no strong dominating path, there is a component $G'$ in $G-P$ with $|G'| \geq  2$. Because $G$ is $2$-connected, we can find two vertices $x_1,x_t\in V(G')$ such that $x_1 v_i, x_t v_j \in E(G)$, $i<j$, and $Q=x_1 x_2 \cdots x_t$ is a longest path connecting $x_1$ and $x_t$ in $G'$. Obviously, $t \geq  2$. By the maximality of $P$, we have
$$\max\{i-1,j-i-1,p-j\}\geq t ,$$
which implies $p\geq 3t+2$. Since $k\leq 9$, we have $p\leq 8$, and hence $t=2$, $p=8$, $k=9$, $i=3$, $j=6$.

Now we show that $G'=K_2$.  If $G'\not=K_2$, then there exists some $x\in V(G')-\{x_1,x_2\}$ that is adjacent to $x_1$ or $x_2$, and so $v_1\overrightarrow{P} v_6x_2x_1x$ or $v_8\overleftarrow{P} v_3x_1x_2x$
is a path longer than $P$, a contradiction. Hence $G'=K_2$ and the component of $G-P$ is either $K_1$ or $K_2$.

Since $k=9$, $\delta(G) \geq\delta_k=3$. By the maximality of $P$, we can check that $N(x_1)=\{ x_2, v_3, v_6\}$ and $N(x_2)=\{ x_1, v_3, v_6\}$, which implies that $m\geq 4$. If  $G-P$ contains an isolated vertex $u$, then since $d(u)\geq 3$ and $u$ is not adjacent to $v_1$, $v_8$ and consecutive vertices in $P$, we have $\{v_3,v_6\}\not\subseteq N_P(u)$, and so there are at least two vertices in $N_P(u)\setminus \{ v_3, v_6\}$, which will result in a path longer than $P$ in $G$, a contradiction. Thus all components in $G-P$ are $K_2$. 
By $\delta(G) \geq 3$ and the maximality of $P$, $v_3, v_6$ are neighbors of $v_1$, $v_2$, $v_4$, $v_5$, $v_7$ and $v_8$. Hence we have $G=I_2 \vee \frac{n-2}{2} K_2$ if $v_3v_6\notin E(G)$, or $G=K_2 \vee \frac{n-2}{2} K_2$ if $v_3v_6\in E(G)$ and $m\geq 5$, that is, (4) holds.

Therefore, we complete the proof of Theorem \ref{max}.$\hfill \qed$

\section{Generalized Tur\'an number}\label{sec4}

\hspace{1.5em}In this section, we consider the generalized Tur\'an number problem. First we need the following lemmas.
\begin{lem}\label{lem4}
	Let $G$ be a $\{P_k, K_m\}$-free graph of order $n$, $H$ be a $K_{m-1}$-free graph of order $\delta_k$ and $r\le \delta_k+1$. If $$G\subseteq \begin{cases}
		H \vee I_{n-\delta_k},	& \text{ if } k \text{ is even},  \\
		H \vee \left( I_{n-\delta_k-2} \cup K_2 \right),	& \text{ if } k \text{ is odd},
	\end{cases}$$ then there exists an integer $n_1$, such that when $ n > n_1$, $N_r(G) \leq N_r(H_n(m,k))$, with equality if and only if $G\cong H_n(m,k)$, or $G\cong H_n^-(m,k)$ with $m-2\le \delta_k\le 2m-5$ and $k$ odd.
\end{lem}

\begin{proof}
	
If $\delta_k+2<m< k$, then since $G$ is a subgraph of $H_n(m,k)$,  $N_r(G) \leq N_r(H_n(m,k))$.
If $m \leq \delta_k+2$ and $G \subseteq H \vee I_{n-\delta_k}$, then by Lemma \ref{ThKr},
	\begin{align*}
		N_r(G) \leq & N_r(H) + N_{r-1} (H) \cdot (n-\delta_k) \\
		\leq & N_r \left(T\left(\delta_k, m-2 \right) \right) + N_{r-1} \left(T\left(\delta_k, m-2 \right) \right) \cdot (n-\delta_k) 
		= N_r( H_n(m,k)).
	\end{align*}
	All equalities hold if and only if $G\cong H_n(m,k)$.
	
	Now, assume that $m \leq \delta_k+2$ and $G \subseteq H \vee \left( I_{n-\delta_k-2} \cup K_2 \right)$. In this case, $k$ is odd. Let $e$ denote the only edge in $G-V(H)$.
	First we suppose $H= T\left(\delta_k, m-2 \right)$. Then since $G$ is $K_m$-free, there exists one partite $V_1$ of $T\left(\delta_k, m-2 \right)$ such that any vertex of $V_1$ is adjacent to at most one end of $e$, which leads to a reduction in the number of $K_r$ in $G$ of at least $|V_1| \cdot N_{r-2} ( T\left(\delta_k-|V_1|, m-3 \right))$ compared to $T\left(\delta_k, m-2 \right) \vee I_{n-\delta_k} $. Note that the number of $K_r$ containing $e$ in $G$ is at most $N_{r-2} ( T\left(\delta_k-|V_1|, m-3 \right))$.
	When $r\geq 3$, we have
 \begin{align*}
&N_r(H_n(m,k)) -N_r(G)\\
\geq &|V_1|\cdot N_{r-2} ( T\left(\delta_k-|V_1|, m-3 \right)) -N_{r-2} ( T\left(\delta_k-|V_1|, m-3 \right))\geq 0.
\end{align*}
When $r=2$, then
$$e(G)\leq e\left(T\left(\delta_k, m-2 \right) \vee \left( I_{n-\delta_k-2} \cup K_2 \right) \right)-|V_1|\le  e(H_n(m,k)).$$
	The equalities hold if and only if $|V_1|=1$, that is, we have $m-2\le \delta_k\le 2m-5$ and $G\cong H_n^-(m,k)$.
	
	If $H\neq T\left(\delta_k, m-2 \right)$, then by Lemma \ref{ThKr}, $N_{r-1} (H) < N_{r-1} (  T\left(\delta_k, m-2 \right) )$. Thus there exists an integer $n_1$, such that when $ n > n_1$, we have
	\begin{align*}
		N_r(G)\leq &  N_{r-1} (H) \cdot (n-\delta_k)+ N_r(H) + N_{r-2}(H) \\
		<&  N_{r-1} \left(T\left(\delta_k, m-2 \right) \right) \cdot (n-\delta_k) +N_r \left(T\left(\delta_k, m-2 \right) \right) \\
		=& N_r( H_n(m,k)).
	\end{align*}
	
	The proof is complete.	
\end{proof}

\begin{lem}\label{Lem41}Let $2\le r\le \delta_k+1$, we have

{\rm  (1)} $N_r( G_i(n,k) ) \leq  N_r(H_n(m,k) )$ for $i=1,2,3$,  with equality if and only if $r=2$ and $G_i(n,k)\cong H_n(m,k)$ for $i=1,3$, or $G_2(n,k)\cong S_{a,n-a}$ with $k=5$, $m=3$;

{\rm (2)}  $N_r( G_i(n) ) < N_r(K_2\vee I_{n-2})$ for $i=4,5$;

{\rm (3)} $N_r(I_2 \vee \frac{n-2}{2} K_2 ) < N_r(K_3 \vee I_{n-3})$ and $N_r(K_2 \vee \frac{n-2}{2} K_2 ) < N_r(K_3 \vee I_{n-3})$ for $n>4$.
\end{lem}
\begin{proof}
	By simple calculation, one can get (2) and (3). Now we show that (1) holds.
	
	 Recall $\delta_k \leq m-2$ by the definition of $G_i(n,k)$, $H_n(m,k)=K_{\delta_k}\vee I_{n-\delta_k}$ if $k$ is even and $m>\delta_k+2$, 
  $H_n(m,k)= K_{\delta_k} \vee ( I_{n-\delta_k-2}\cup K_2)$ if $k$ is odd and $m>\delta_k+2$, and $H_n(m,k)=T(\delta_k,m-2)\vee I_{n-\delta_k}=K_{\delta_k}\vee I_{n-\delta_k}$ if $m=\delta_k+2$.
	
Note that $G_1(n,k)=K_1\vee tK_{\delta_k}$, to show that $N_r(G_1(n,k)) \leq N_r \left(H_n(m,k) \right)$,  it suffices to prove
	\begin{equation}\label{eqb1}
		N_r(G_1(n,k))=t\cdot N_r \left( K_{ \delta_k+1}  \right)\le N_r \left( K_{\delta_k} \vee I_{(t-1) \delta_k+1}  \right).
	\end{equation}
It is clear that (\ref{eqb1}) holds when $t=1$. For $t\geq 2$, the inequality (\ref{eqb1}) is equivalent to
\begin{equation*}\label{eqb2}
	(t-1)\cdot N_r \left( K_{ \delta_k+1}  \right)\leq (t-1) \delta_k \cdot N_{r-1} \left( K_{\delta_k} \right).
\end{equation*}
Since $2\leq r\leq \delta_k+1$, we have $\delta_k \geq 1\ge \frac{1}{r-1}$, which implies
$$(t-1){\delta_k+1\choose r}=(t-1)\cdot \frac{ \delta_k+1 }{r} \binom{\delta_k}{r-1} \leq  (t-1)  \delta_k  \cdot \binom{\delta_k}{r-1}.$$
Hence $N_r( G_1(n,k) ) \le  N_r(H_n(m,k) )$, equality holds if and only if $\delta_k=1= \frac{1}{r-1}$, that is, $r=2,~k=4,5,~m=3$ and $G_1(n,k)\cong H_n(m,k)$.

Suppose that $G_2(n,k)$ contains $t$ blocks. In this case, $k$ is odd by the structure of $G_2(n,k)$ in Lemma \ref{lem_block}. Then by (\ref{eqb1}),
 $$N_r(G_2(n,k))=N_r(G_1(n-1,k)) \le  N_r \left(K_{\delta_k} \vee I_{n-\delta_k-1} \right)< N_r \left(H_n(m,k) \right)$$ if $r\geq 3$, and if $r=2$, then
\begin{align*}
N_2(G_2(n,k))=&N_2(G_1(n-1,k))+1 \le  N_2 \left(K_{\delta_k} \vee I_{n-1-\delta_k } \right) +1 \\
=& N_2 \left(K_{\delta_k} \vee I_{n-\delta_k}\right) +1-\delta_k\le N_2(H_n(m,k)),
 \end{align*}
equalities hold if and only if $k=5, m=3$, that is $G_2(n,k)$ is a double star.

Suppose $G_3(n,k)$ has $t$ blocks. In this case, $k$ is odd. When $m>\delta_k+2$, $H_n(m,k)= K_{\delta_k} \vee ( I_{n-\delta_k-2}\cup K_2)$. By the definition of $G_3(n,k)$ and (\ref{eqb1}), we have
\begin{align*}
N_r(G_3(n,k))&\le N_r(K_1 \vee ( (t-1) K_{\delta_k} \cup K_{\delta_k+1} )) \\
&= N_r(G_1(n-1,k))+\binom{\delta_k+1}{r-1}\\
&\le N_r(K_{\delta_k}\vee I_{n-1-\delta_k})+\binom{\delta_k}{r-1}+\binom{\delta_k}{r-2}\\
&= N_r(H_n(m,k)),
\end{align*}
equalities hold if and only if $r=2$, $k=5$, $m=4$ and $G_3(n,k)\cong H_n(m,k)$.
When $m = \delta_k+2$, we have $H_n(m,k)= K_{\delta_k} \vee ( I_{n-\delta_k})$ and the block $B$ of order $\delta_k+2$ in $G_3(n,k)$ must be a proper subgraph of $K_{\delta_k+2}$. Let $u$ be a vertex in $B$ such that $u$ is not adjacent to some other vertex in $B$. Because the number of $K_r$ containing $u$ is at most $N_{r-1} \left(K_{\delta_k} \right)$ in $B$, we have
\begin{align*}
N_r(G_3(n,k))& \le N_r(G_3(n,k)-u )+\binom{\delta_k}{r-1}\\
& \le N_r(G_1(n-1,k))+\binom{\delta_k}{r-1}\\
&\le N_r(K_{\delta_k}\vee I_{n-1-\delta_k})+\binom{\delta_k}{r-1}= N_r(H_n(m,k)),
\end{align*}
equalities hold if and only if $r=2$, $k=5$, $m=3$, which implies the block $B$ is a proper subgraph of $K_{\delta_k+2}=K_3$, a contradiction.

The proof is complete.
\end{proof} 

In the proofs of Theorems \ref{gener} and  \ref{nonconnect}, we define $\ex(n,K_r, K_m)=\binom{n}{r}$ and $T(n,m)=K_n$ if $m\geq n$ for convenience.
\vskip 2mm
\noindent\textbf{Proof of Theorem \ref{gener}}.
Let $G$ be a connected $\{P_k, K_m \}$-free graph of order $n$ with
$N_r(G)=\ex_{con}(n, K_r, \{P_k, K_m \})$. Then
\begin{equation}\label{eq44}
 N_r(G) \geq N_r(H_n(m,k)).
\end{equation}

Recall that $\delta_k= \lfloor \frac{k}{2} \rfloor -1$. We define a process of $G$, called ($\delta_k -1$)-disintegration as follows: for
$j = n$, let $G^n = G$, and for $j < n$, let $G^j$ be obtained from $G^{j+1}$ by deleting a vertex of degree less than $\delta_k$ in $G^{j+1}$ if such a vertex exists. The process terminates at $G^t$ when $\delta(G^t)\geq \delta_k$ or $t=0$. We call $G^t$ the $ \delta_k $-core of $G$. 

Now we show that it suffices to consider the case where $G^t$ is connected. For $j>t$, if $G^j$ is $2$-connected, then $G^{j-1}$ is still connected after deleting a vertex of degree less than $\delta_k$ in $G^j$. Suppose that $G^j$ is not $2$-connected. To ensure that $G^{j-1}$ is connected, we delete a vertex of degree less than $\delta_k$ in the end blocks of $G^j$ that is not a cut vertex. If there is no such vertex, then by Lemma \ref{lem_block} and Lemma \ref{Lem41}, $N_r(G^j) \leq N_r(H_j(m,k))$. Since deleting one vertex during the process of ($\delta_k -1$)-disintegration destroys at most $\ex(\delta_k-1,K_{r-1},K_{m-1})$ copies of $K_r$, we have
\begin{align*}
	N_r(G) & \leq N_r(G^j)+(n-j) \cdot \ex(\delta_k-1,K_{r-1},K_{m-1})\\
& < N_r(H_j(m,k))+(n-j) \cdot N_{r-1} (T(\delta_k,m-2)) \leq N_r(H_n(m,k)).
\end{align*}
Thus we only need to consider the case where $G^t$ is connected.

If $t=0$, then since $n$ is sufficiently large, we have
\begin{align*}
	N_r(G) & \leq \ex(\delta_k, K_r, K_m) +(n-\delta_k) \cdot \ex(\delta_k-1, K_{r-1}, K_{m-1}) \\
	& < \ex(\delta_k, K_r, K_{m-1}) +(n-\delta_k) \cdot \ex(\delta_k, K_{r-1}, K_{m-1})\leq N_r(H_n(m,k)),
\end{align*}
which contradicts (\ref{eq44}). Thus $\delta (G^t) \geq \delta_k$ and $t\geq \delta_k+1$.

Recall that $r\!\leq\! \min\{m-1, \delta_k+1\}$. Since $N_r(G^{j+1})-N_r(G^j)\!\leq  \ex(\delta_k-1,K_{r-1},K_{m-1})$ and $N_r(H_{j+1} (m,k)) -N_r(H_j(m,k)) \geq N_{r-1}(T(\delta_k,m-2))=\ex(\delta_k,K_{r-1},K_{m-1})$ for $t\leq j \leq n-1$, we have
$$N_r(G^j)- N_r(H_j(m,k)) \geq N_r(G^{j+1}) - N_r(H_{j+1} (m,k)) +1.$$
Thus
$$N_r(G^t)- N_r(H_t(m,k)) \geq N_r(G^{n}) - N_r(H_n (m,k)) +(n-t) ,$$
which implies that
$$N_r(G^t) \geq n-t+ N_r(H_t(m,k)) \geq n-\delta_k .$$
On the other hand, since $G^t$ is $P_k$-free, by Theorem \ref{luo}, we have
$$N_r(G^t)\leq  \frac{t}{k-1} \binom{k-1}{r}.$$
Hence, $$t\geq \frac{(n-\delta_k)(k-1)}{\binom{k-1}{r}}.$$
Let $n> \binom{k-1}{r}\cdot \frac{n_1}{k-1} +\delta_k$, then $t\geq n_1$. By Theorem \ref{max}, Lemmas \ref{lem4} and  \ref{Lem41}, we have $N_r(G^t) \leq N_r(H_t(m,k))$. Hence,
\begin{align*}
&N_r(G)\leq N_r(G^t)+(n-t) \cdot N_{r-1} (T(\delta_k-1,m-2))\\
\leq& N_r(G^t)+(n-t) \cdot N_{r-1} (T(\delta_k,m-2)) \leq N_r(H_n(m,k)),
\end{align*}
 with equality if and only if $n=t$ and $G\cong H_n(m,k)$, $G\cong H_n^-(m,k)$ with $m-2\le \delta_k\le 2m-5$ and $k$ odd, or $G\cong S_{a,n-a}$ with $k=5$, $m=3$ and $r=2$.

The proof is complete. $\hfill \qed$
\vskip 2mm
\noindent\textbf{Proof of Theorem \ref{nonconnect}}. We apply ($\delta_k -1$)-disintegration to $G$. Let $G^t$ be the $ \delta_k $-core of $G$ and $G_i (1 \leq i \leq s)$ be the connected components of $G^t$ with $\ell_i = |G_i|$.

Let $n_1$ be an integer such that Lemma \ref{lem4} holds.

If $\ell_i > n_1$, then by Theorem \ref{max}, Lemmas \ref{lem4} and  \ref{Lem41},
\begin{equation}\label{eq41}
N_r(G_i) \leq N_r(H_{\ell_i}(m,k)).
\end{equation}

If $k\leq \ell_i \leq n_1$, then by Theorem \ref{max}, Lemmas \ref{Lem41} and  \ref{ThKr},
	\begin{align*}
		&N_r(G_i)  \leq N_r \left( 	T \left( \delta_k, m-2 \right)  \vee \left( I_{\ell_i -\delta_k-2} \cup K_2 \right)  \right) \\
		=&  N_{r-1} \left(T \left(\delta_k, m-2 \right) \right)   \cdot  (\ell_i-\delta_k)
		 +N_{r} \left(T \left(\delta_k, m-2 \right) \right) +N_{r-2} \left(T \left(\delta_k, m-2 \right) \right)  \\
		< & N_{r-1} \left(T \left(\delta_k, m-2 \right) \right)   \cdot  \ell_i.
	\end{align*} 
The last inequation follows from the fact that $\delta_k  \cdot N_{r-1} \left(T \left(\delta_k, m-2 \right) \right) > N_{r} \left(T \left(\delta_k, m-2 \right) \right) \\ + N_{r-2} \left(T \left(\delta_k, m-2 \right) \right)$ since $\delta_k  \cdot N_{r-1} \left(T \left(\delta_k, m-2 \right) \right)$ can be viewed as the number of combinations of any vertex and any $K_{r-1}$ in $T \left(\delta_k, m-2 \right)$ and removing any vertex of $K_{r-1}$ will obtain a $K_{r-2}$ while adding any vertex to $K_{r-1}$ yields a $K_r$.
Since $\ell_i \leq n_1$, there exists a small constant $\varepsilon$ such that
\begin{equation}\label{eq42}
	N_r(G_i) \leq (N_{r-1} (T(\delta_k, m-2)) -\varepsilon ) \cdot \ell_i.
\end{equation}

If $\ell_i <k$, then 
\begin{equation}\label{eq43}
 N_r(G_i) \leq   \frac{ N_r(T(\ell_i, m-1) )}{\ell_i} \cdot \ell_i  \leq \frac{ N_r(T(k-1,m-1))}{k-1} \cdot \ell_i ,
\end{equation}
where the last inequation follows from the fact that $\frac{ N_r(T(\ell_i, m-1) )}{\ell_i} \leq \frac{ N_r(T(\ell_i+1, m-1) )}{\ell_i+1} $ since the average number of $K_r$ containing a vertex $v$ in $T(\ell_i+1, m-1)$ is not less than that in $T(\ell_i, m-1)$ (Note that $T(\ell_i+1, m-1)$ can be obtained by adding a new vertex $u$ to a part of size $\lfloor \frac{\ell_i}{m-1} \rfloor$ in $T(\ell_i, m-1)$ and the number of $K_r$ containing $u$ in $T(\ell_i+1, m-1)$ is not less than the average number of $K_r$ containing a vertex $v$ in $T(\ell_i, m-1)$ .).

When $N_{r-1} ( T(\delta_k, m-2)) >\frac{ N_r(T(k-1,m-1))}{k-1}$, for $\ell_i <k$, by (\ref{eq43}), we have
\begin{equation}\label{eq4}
	N_r(G_i) \leq \left(N_{r-1} \left(T \left(\delta_k, m-2 \right) \right) -\varepsilon \right)\cdot \ell_i.
\end{equation}
If there is at least one $\ell_i >n_1$, then by (\ref{eq41}), (\ref{eq42}) and (\ref{eq4}), we have $N_r(G^t)=\sum_{i} N_r(G_i) \leq   N_r(H_t(m,k))$, with equality if and only if $G^t$ is connected with $t>n_1$. If $\ell_i \leq n_1$ for $1 \leq i \leq s$, then
$$N_r(G^t)=\sum_{i} N_r(G_i) \leq \sum_{i} \left(N_{r-1} \left(T \left(\delta_k, m-2 \right) \right) -\varepsilon \right)\cdot  \ell_i =t\cdot\left(N_{r-1} \left(T \left(\delta_k, m-2 \right) \right) -\varepsilon \right).$$ Thus there exists an $n_2$ such that when $t >n_2$, $N_r(G^t) < N_r(H_t(m,k))$. By a discussion of ($\delta_k -1$)-disintegration similar to that as in the proof of Theorem \ref{gener}, we have $N_r(G) \leq N_r(H_n(m,k))$, and the equality holds if and only if $n=t$ and $G\cong H_n(m,k)$, $G\cong H_n^-(m,k)$ with $m-2\le \delta_k\le 2m-5$ and $k$ odd, or $G\cong S_{a,n-a}$ with $k=5$, $m=3$ and $r=2$.

When $N_{r-1} ( T(\delta_k, m-2)) \leq \frac{ N_r(T(k-1,m-1))}{k-1}$, then by (\ref{eq41}), (\ref{eq42}) and (\ref{eq43}), we have $$N_r(G^t) = \sum_{i} N_r(G_i) \leq  \sum_{i}
\frac{ N_r(T(k-1,m-1))}{k-1} \cdot  \ell_i =\frac{ N_r(T(k-1,m-1))}{k-1} \cdot t.$$ Thus $N_r(G)\leq N_r(G^t)+(n-t) \cdot N_{r-1} (T(\delta_k-1,m-2)) \leq \frac{ N_r(T(k-1,m-1))}{k-1} \cdot n$. When $(k-1) | n$, the equality holds and $\frac{n}{k-1} \cdot T(k-1, m-1)$ is an extremal graph.

The proof is complete. $\hfill \qed$

\vskip 2.5em
\section*{Acknowledgments}

 This research was supported by NSFC under grant numbers  12161141003 and 11931006.  

\vskip 5mm
\noindent\textbf{Data availability statement} \  Data sharing not applicable to this article as no datasets were generated or analyzed during the current study.

\vskip 5mm
\noindent\textbf{Declarations of conflict of interest}\  The authors declare that they have no known competing financial interests or personal relationships that could have appeared to influence the work reported in this paper.


\begin{thebibliography}{}

\bibitem{alon2016} N. Alon, C. Shikhelman, Many $T$ copies in $H$-free graphs, J. Comb. Theory, Ser. B 121 (2016) 146--172.

\bibitem{Balister}P.N. Balister, E. Gy\H{o}ri, J. Lehel, R.H. Schelp, Connected graphs without long paths, Discrete Math. 308 (2008) 4487–4494.

\bibitem{bondy} J.A. Bondy, U.S.R. Murty, Graph Theory, Springer-Verlag, New York, 2008.

\bibitem{DQ} D. Chakraborti, D.Q. Chen, Exact results on generalized Erd\H{o}s-Gallai problems,  Eur. J. Comb. 120 (2024) 103955.

\bibitem{Erdos1959} P. Erd\H os, T. Gallai, On maximal paths and circuits of graphs, Acta Math. Acad. Sci. Hungar 10 (1959) 337--356.

\bibitem{Faudree} R.J. Faudree, R.H. Schelp, Path ramsey numbers in multicolorings, J. Comb. Theory, Ser. B 19 (1975) 150--160.

\bibitem{katona2023} G.O.H. Katona, C. Xiao, Extremal graphs without long paths and large cliques, Eur. J. Comb. 119 (2024) 103807.

\bibitem{Kopylov} G.N. Kopylov, Maximal path and cycles in a graph, Dokl. Akad. Nauk SSSR 234 (1977) 19--21.

\bibitem{luo2017} R. Luo, The maximum number of cliques in graphs without long cycles, J. Comb. Theory, Ser. B 128 (2017) 219--226.

\bibitem{Kang} Y. Liu, L. Kang, Extremal graphs without long paths and a given graph, Discrete Math. 347 (2024) 113988.


\bibitem{saito1999} A. Saito, Long paths, long cycles, and their relative length, J. Graph Theory 30 (1999) 91--99.

\bibitem{turan1} P. Tur\'an, Eine Extremalaufgabe aus der Graphentheorie, Mat. Fiz. Lapok 48 (1941) 436--452. (in Hungarian)

\bibitem{zykov1949} A. Zykov, On some properties of linear complexes, Mat. Sbornik N. S. 24 (1949) 163--188.

\end{thebibliography}
\end{document}